\documentclass{amsart}
\usepackage{amscd,amssymb,amsmath,amsthm}
\usepackage[dvips]{graphicx}
\usepackage[all]{xy}
\theoremstyle{plain}
\newtheorem{definition}{Definition}
\newtheorem{proposition}{Proposition}
\newtheorem{theorem}[proposition]{Theorem}

\newtheorem{lemma}[proposition]{Lemma}

\newtheorem*{proposition*}{Proposition}
\newtheorem*{theorem*}{Theorem}
\newtheorem*{corollary*}{Corollary}
\newtheorem*{lemma*}{Lemma}
\newtheorem*{remark*}{Remark}
\newtheorem*{example*}{Example}

\newcommand{\R}{\mathbb{R}}
\newcommand{\C}{\mathbb{C}}
\newcommand{\Coker}{\mbox{coker\ }}
\newcommand{\coker}{\mbox{coker\ }}

\begin{document}

\title{A simple proof of a theorem of Fukaya and Oh}

\author{Vito Iacovino}

\address{Department of Mathematics, University of California at Irvine, Irvine, CA 92697, USA.}

\email{viacovin@math.uci.edu}

\date{version: \today}


\begin{abstract}
We study the moduli space of pseudo pointed holomorphic disks with boundaries mapped in the zero section of the cotangent bundle of a manifold. We define perturbations of the equation for which it is possible to describe explicitly all the solutions of the problem in terms of Morse graphs on the manifold. 
In particular, this proves that the $A_\infty$ structure of the zero section of the cotangent bundle is equivalent to the Morse $A_\infty$ structure of the base manifold.  
\end{abstract}

\maketitle

\section{Introduction}

Let $M$ be a Riemannian manifold. 
In \cite{FO} Fukaya and Oh study the moduli space of pointed pseudo-holomorphic disks bounding (perturbations of) the zero section $L$ of $T^*M$.
More precisely a function on $M$ is assigned on any boundary component of the disk. The boundary component is mapped to the graph of the differential of the associated function. Fukaya and Ho consider the "adiabatic limit" of this problem, namely they rescale the functions by a small real number $\varepsilon>0$. Under generic conditions on the functions, for $\varepsilon$ that goes to zero the $J$-holomorphic disks degenerate to the Morse trees of the functions.   



The construction of \cite{FO} generalizes the construction of Floer in \cite{Fl}. Floer showed a chain level isomorphism between the Lagrangian Floer complex of $L$ and the Morse complex of $M$. 
Given a Morse function $f$ on $M$, the pseudo-holomorphic strips that are bounded by $L$ and the graph of $\varepsilon df$ are in correspondence with the gradient flow lines of $f$. More precisely, there exists a complex structure on $T^*M$ for which it is possible to describe explicitly all these strips in terms of the Morse gradient flows of $f$.   


One main complication of the problem of \cite{FO} with respect to the problem of \cite{Fl} is that it is not possible to find explicit solutions. Therefore, in order to construct from a Morse tree a pseudo-holomorphic disk it is necessary to find an approximative solution and use an iterative process. This process involves a delicate analysis of the estimates when $\varepsilon$ goes to zero.

The reason why it is interesting to consider these perturbations of the problem is that there are algebraic structures associated with the moduli space of solutions  (such as the Fukaya $A_\infty$ structure) that are independent of the perturbation. However it is possible to consider more general perturbation of the problem (see \cite{S}). In particular it is possible (and in some sense more elementary) to consider inhomogeneous perturbations. 


We define inhomogeneous perturbations of the pseudo-holomorphic equation for which it is possible to describe explicitly all the solutions.
More precisely, we associate a Floer datum to a set of functions on $M$ and we construct an inhomogeneous perturbation compatible with this datum. After rescaling the functions we are able to write all the solutions of the problem in terms of Morse trees on $M$. Furthermore, if the Morse trees are transversal (which holds for a generic choice of the functions) we prove the Floer transversality of the solutions.

We consider the space of associated Morse graphs as a stratified space. The stratum of codimension $k$ is given by the graphs with $k$ edges of length zero. For a generic choice of the functions any stratum will be transversal. For the principal stratum there is a one to one correspondence between solutions and graphs. To a graph in the codimension $k$ stratum corresponds a family of dimension $k$ of solutions. 

{\bf Acknowledgements. } We are grateful to P. Seidel for insightful discussions.

\section{Morse trees}

Let $T$ be a tree. Denote by $E(T)$ and $V(T)$ the sets of edges and vertices of $T$. Let $E^{ex}(T)$ and $E^{in}(T)$ be the sets of external and internal edges respectively. Let $H(T)$ be the set of half edges or edges with a starting point. In $H(T)$ there are two elements for any internal edge and one for any external edge.  

For any vertex $v \in V(T)$ let $H(v)$ be the set of half edges starting in $v$ and denote by $|v|$ the cardinality of $H(v)$. A ribbon structure on $T$ assigns to any vertex $v$ a cyclic order of $H(v)$. 
 
A Riemannian metric on $T$ assigns a metric on each edge of $T$ such that any external edge has length infinity. The moduli space of Riemannian metrics of $T$ modulo isometry is $(\R^+)^{E^{in}(T)}$. We denote by $L(e)$ the length of the edge $e$.

Let $M$ be a manifold endowed with a Riemannian metric $g$.
Suppose that to any pair of boundary components of the ribbon tree $T$ Morse function on $M$ is assigned on $M$. In particular any oriented edge $e$ is assigned a Morse function $f_e$.

A Morse gradient tree assigns to any oriented edge $e$ of $T$ a map $\gamma_e$ from an interval of length $L(e)$ to $M$
such that
\begin{equation} \label{morseeq}
\dot \gamma_e - \nabla f_e (\gamma_e)=0.
\end{equation}
The starting point of these curves has to agree for edges starting from the same vertex.



Fix for any external edge $e$ the corresponding external vertex $p_e$. Denote by $\mathcal{G}^T(g,f,p)$ the moduli space of Morse gradient trees. 

For an external edge $e$ let $M_e$ be the unstable manifold of $f_e$ and let $\pi_e : M^{H(\gamma)} \rightarrow M$ be the projection associated to the edge $e$. 

For any internal edge $e$ let $M_e$ be the submanifold of $M^2$ given by 
$$ M_e= \{  (x, \phi^e_t(x) ) | x \in M \}$$
the pairs of the initial and final points of the Morse trajectories of $f_e$. Let $\pi_e : M^{H(\gamma)} \rightarrow M^2$ the projection corresponding to the vertices attached to $e$. 


Define the submanifold $\mathcal{E}^T$ of $M^{H(T)}$ by
$$\mathcal{E}^T= \bigcap_{e \in E(T)} \pi_e^{-1} (M_e).$$

For any vertex $v$ let $\pi_v: M^{H(T)} \rightarrow M^{H(v)}$ be the projection associated to the edges attached to $v$. Let $\Delta_v \in M^{|v|}$ be the diagonal.
Define
$$\mathcal{V}^T = \bigcap_{v \in V(T)} \pi_v^{-1} \Delta_v .$$

We have 
$$ \mathcal{G}^T = \mathcal{V}^T \cap \mathcal{E}^T  . $$


One says that the set of functions $f$ are transversal if $\mathcal{V}^T$ and $\mathcal{E}^T$ intersect transversally for every tree $T$. It is a standard fact that this holds for a generic choice of functions.

\section{Inhomogeneous pseudo-holomorphic disks}

Let $\mathcal{T}_k$ be the moduli space of disks with $k$ marked points on the boundary. 

Fix on the $\mathcal{T}_k$ a consistent universal choice of strip-like ends (see section 9g of \cite{S}). An easy consequence of Lemma 9.2 of \cite{S} is the following

\begin{lemma} 
Fix $d \geq 3$. For each $k \leq d$ there exists a bounded region $\mathcal{B}_k$ of $\mathcal{T}_k$ such that any disk in $\mathcal{T}_d$ can be obtained in unique way gluing disks in $\cup_k \mathcal{B}_k$.    
\end{lemma}
The lemma provides a decomposition
\begin{equation} \label{decomposition}
\mathcal{T}_d=\bigcup_T \Pi_{v \in V(T)} \mathcal{B}_{|v|} \times {\R^{E^{in}(T)}}.
\end{equation}
We will refer to the components of this decomposition as vertex regions and strips regions. 


This decomposition associates to any punctured disk a ribbon tree. The boundary components of the disk correspond to the boundary components of the ribbon tree.


In the following we endow $T^*M $ with the complex structures induced from the Levi-Civita connection of the metric on $M$.

Now we want to define the moduli space of inhomogeneous pseudo-holomorphic disks. In order to define the Hamiltonian perturbation of the standard d-bar equation we need to fix a cut off function $\rho$ that is zero in the vertex regions and equal to one inside the strip regions. This can be done in the following way. On an internal strip $[0,1] \times [0,l]$, $\rho=\rho_l$ depends smoothly on the length $l$ of the strip.
Pick a real increasing function $\phi(t)$ equal to zero for $t < 0$ and equal one for $t > 1$. We put $\rho_l(s) = \phi(l) \phi(s) \phi(l-s)$. 
On an external strip $[0,1] \times [-\infty, 0]$ there is no parameter $l$. One can fix the cutoff using $\rho(s)= \phi(-s)$. 

\begin{definition} \label{problem}
Suppose that any pair of boundary components of the disks in $\mathcal{T}_k$ is associated to a Morse function on $M$. In particular, any edge $e$ of the decomposition  (\ref{decomposition}) it is associated to a Morse function $f_e$.  
Fix $\varepsilon>0$.

Let $\mathcal{M}_J(T^*M, \varepsilon f, p)$ be the moduli space of maps $u: \Sigma \rightarrow X$ such that for any boundary component $i$ 
$$u(\partial_i \Sigma) \subset L ,$$
on any vertex region
$$\bar \partial u =0 ,$$
and on an internal strip $e$ 
$$\partial_s u + J(u) (\partial_t u - \varepsilon \rho X_e (u)) =0  $$
where $X_e$ is the Hamiltonian vector field of $f_e$ (here we consider $f_e$ as a function on $T^*M$ constant along the fibers).  
\end{definition}

On each strip we can rewrite the equation as
\begin{equation} \label{strip}
\partial_s u + J(u) \partial_t u - \varepsilon \rho \nabla f_e (u) =0 .
\end{equation}

We will associate solutions of the inhomogeneous pseudo-holomorphic equation to the Morse graphs. More precisely, for a tree in the substratum of codimension $k$ we will construct a contractible family of solutions of dimension $k$ .




Fix a Morse tree in $\mathcal{M}_g^T (M, f , p)$ and a point in $\Pi_{v \in V(T)} \mathcal{B}_{|v|}$. We define a Riemannian disk $\Sigma$ using $(\ref{decomposition})$.To any edge of length $R$ there corresponds a strip of length $l$ such that 
$$R= \varepsilon \int_0^l \rho_l(s)ds .$$ 


Define a map $u: \Sigma \rightarrow T^*M $ as follows. $u$ maps any vertex region to the corresponding vertex of the Morse tree. In the strip of the edge $e$, $u$ is given by 
\begin{equation} \label{stripsolution}
u(t,s) = \gamma_e( l(s))
\end{equation}
where $l(s)= \varepsilon \int_0^s \rho $.

This construction defines a map 
\begin{equation} \label{solution}
\bigcup_{T} \mathcal{M}_g^T (M, f , p) \times \Pi_{v \in V(T)} \mathcal{B}_{|v|} \rightarrow  \mathcal{M}_J^{\varepsilon}(T^*M,  f, p).
\end{equation}
 

\begin{theorem} 
For $\varepsilon$ small enough, the map (\ref{solution}) is bijective. Taking a smaller $\varepsilon$ if necessary, the solutions of the inhomogeneous pseudo-holomorphic equation are transverse if the Morse trees are transverse. 
\end{theorem}

\section{Surjectivity}


Let $u : \Sigma \rightarrow T^*M$ be a solution of (\ref{problem}). Let $T$ be the tree associated to $\Sigma$ by the decomposition $(\ref{decomposition})$.

Write $u=(q,p)$ where $q \in  M$ and $p \in T^*_q M$. The symplectic form $\omega$ on $T^*M$ is given by the differential of the one form $\theta= \langle p, dq \rangle$. 

\begin{lemma} 
For any strip $S$ 
\begin{equation} \label{se}
\int_{\partial S} u^*(\theta) \geq 0 
\end{equation}
and for any vertex region $V$ 
\begin{equation} \label{ve}
\int_{\partial V} u^*(\theta) \geq 0 .
\end{equation}
\end{lemma}

\begin{proof}
Fix a strip $e$. In the following we will omit the suffix $e$ in the notation. Equation (\ref{strip}) can be rewritten as 
\begin{equation} \label{orr}
\partial_s q + \nabla_t p = \varepsilon \rho \nabla f (q)
\end{equation}
\begin{equation} \label{ver}
\partial_t q - \nabla_s p =0 .
\end{equation}
Define the function
$$\beta(s)= \frac{1}{2}\int_0^1 |p(t,s)|^2 dt.$$
Using (\ref{orr}) and (\ref{ver}) we have
$$\dot \beta = \int_0^1 \langle p, \partial_t q \rangle dt $$
$$
\begin{array} {ll}
\ddot{\beta} &= \int_0^1 (\langle \nabla_s p ,\partial_t q \rangle  +  \langle p, \nabla_s \partial_t q  \rangle) dt 
= \int_0^1 ( |\nabla_s p |^2  -  \langle  p,  \nabla_t (\varepsilon \rho \nabla f - \nabla_t p) \rangle) dt  \\ 
& = \int_0^1( |\nabla_s p|^2 + | \nabla_t p |^2 - \varepsilon \rho \langle p , \nabla(\nabla f) \partial_t q \rangle) dt . \\
\end{array}
$$
Therefore
$$ \ddot{\beta} \geq \int_0^1( |\nabla_s p |^2 + |\nabla_t p|^2 ) dt - C \varepsilon \rho  \int_0^1 |\nabla_s p | |p| dt $$
and in particular $\ddot{\beta} $ is not negative for small enough $\varepsilon$. From this follows (\ref{se}) (observe that if $e$ is external $\beta$ and its derivative goes to zero at infinity).

The inequality (\ref{ve}) follows from
$$\int_{\partial V} u^*(\theta)= \int_V u^*(\omega) \geq 0 .$$

\end{proof}


It follows that in the identity 
$$ \sum_V  \int_{\partial V}  u^*(\theta) + \sum_S  \int_{\partial S} u^*(\theta) =0 $$
all the terms in the sums are not negative, hence they must be zero. In other words, the energy of the curve on any vertex region has to be zero and $\dot \beta$ on any strip has to be constant. In particular $u$ is constant on any vertex region. Hence $\dot \beta$ is zero. It follows that $\beta$ (and therefore $p$) is identically zero on any strip. Now, on any strip $e$, (\ref{orr}) implies that $u$ does not depend on $t$ and (\ref{ver}) implies that $\partial_s q = \varepsilon \rho \nabla f_e(q)$. Therefore $q=\gamma_e(\varepsilon l(s))$ for some solution $\gamma_e$ of equation (\ref{morseeq}). These data define a Morse tree of $T$. It follows that $u$ is in the image of (\ref{solution}).

\section{Transversality}

In this section we will prove that, for $\varepsilon$ small enough, the solutions of problem \ref{problem} are transversal if the Morse transversality holds.

In order to see the relation between Morse and Floer transversality we need to reformulate the Morse transversality in terms of the linearization of the Morse equations (\ref{morseeq}). 

The linearization of the Morse tree equation gives a map 
\begin{equation} \label{graphlinear}
D_0 : L^2_1(\gamma^*(TM)) \oplus \R^{E^{in}(T)} \rightarrow L^2(\Omega^1 (\gamma^*(TM))).
\end{equation}
Here $\gamma^*(TM)$ denotes the unions of the $\gamma_e^*(TM)$ attached to the vertices. A section in $L^2_1(\gamma^*(TM))$ is defined as a section in $L^2_1(\gamma_e^*(TM))$ on each $e \in E$ so that they are compatible at the vertices.

In order to define the linearization we need to fix for any internal edge $e$ a real function $\chi_e$ with compact support. 
The restriction of $D_0^e$ on an internal edge $e$ is given by 
\begin{equation} \label{morselinear}
D_0^e( \xi, \lambda) = \nabla_t \xi - \nabla (\nabla f_e) (\xi) - \lambda_e \chi_e \dot \gamma.
\end{equation}
Here $\lambda_e$ is the infinitesimal generator of the rescaling of the metric $\langle , \rangle \mapsto e^{2 \lambda_e \chi_e} \langle , \rangle$.
For an external edge $e$ there is not the modular parameter in (\ref{morselinear}).

\begin{lemma} 
$$ \ker D \cong T \mathcal{G} .$$
\end{lemma}
\begin{proof}
Observe first that for any $e \in E$
\begin{equation} \label{tangent0}
\ker D_e \cong TM_e. 
\end{equation}
This can be seen as follows. Suppose first that $e$ is internal. The space of solutions of equation (\ref{morselinear}) with $\lambda_e=0$ is a vector space of dimension $n$. A solution of (\ref{morselinear}) with $\lambda_e=1$ is given by the vector $\psi_e \dot \gamma_e$, where $\psi_e$ is a primitive of $\chi_e$. The linear map $(\xi_e,\lambda_e) \mapsto (\xi_e(0), \xi_e(l_e))$ gives (\ref{tangent0}).
If $e$ is external (\ref{tangent0}) is given by the map $\xi_e \mapsto \xi_e(0)$.

Now let $\xi \in \ker D$. Observe that $\xi$ has to be smooth. Define $v_\xi \in TM^H$ the vector whose components in $e \in H$ is given by the value of $\xi$ at the starting point of $e$. It is immediate to see that $v \in T \mathcal{V}$.
From (\ref{tangent0}) follows that $v \in T \mathcal{E}$.


\end{proof}

\begin{lemma} 
$$ \Coker D = (T\mathcal{V} + T\mathcal{E})^\bot.$$
In particular the functions $f_i$ are transversal if and only if $D$ is surjective. 
\end{lemma}
\begin{proof}

On any $e \in E^{in}(T)$, the adjoint $D_e^*$ of $D_e$, satisfies 
\begin{equation} \label{adjoint}
\langle D_e (\xi,\lambda), \eta \rangle + \langle (\xi,\lambda), D_e^* \eta \rangle = \langle \xi(L(e)), \eta(L(e))(\partial_t) \rangle - \langle \xi(0), \eta(0)(\partial_t) \rangle .   
\end{equation}
For $e \in E^{ex}(T)$ we have
\begin{equation} \label{adjoint1}
\langle D_e \xi, \eta \rangle + \langle \xi, D_e^* \eta \rangle = \langle \xi(0), \eta(0)(\partial_t) \rangle.   
\end{equation}

Take $\eta \in L^2(\Omega^1(T^*(TM)))$ in $\coker D$, that is
$$ \langle D (\xi,\lambda), \eta \rangle =0$$ 
for any $\xi \in L^2_1(T^*(TM))$ and $\lambda \in T\mathcal{L}$. 
Using equations (\ref{adjoint}) and (\ref{adjoint1}) for $\xi$ having compact support, $\eta$ satisfies $D_e^*\eta=0$ for any edge $e$. In particular $\eta$ is smooth. 

For any $e \in H$ evaluate $\eta$ on the positive unit vector in the starting point of $e$. This defines a vector $v_\eta \in TM^H$. 

Let $v \in TM_e$ for $e \in E$. Apply equations (\ref{adjoint}) or (\ref{adjoint1}) where $(\xi,\lambda)$ or $\xi$ is the vector associated to $v$ by (\ref{tangent0}). This implies that $\langle v, v_\eta \rangle =0$. Therefore $v_\eta \in T\mathcal{E}^\bot$.

Let $v \in T\mathcal{V} $ and let $\xi \in L^2_1(T^*(TM))$ such that the value at the vertices is given by $v$. From (\ref{adjoint}) and (\ref{adjoint1}) we have $ \langle v, v_{\eta} \rangle= \langle D (\xi,0), \eta \rangle = 0 $. Therefore $v_\eta \in T\mathcal{V}^\bot$. 
\end{proof}

Now we consider the linearization of the inhomogeneous pseudoholomorphic equation (Definition \ref{problem})
$$D : L^p_1(u^*(T(T^*M))) \oplus \R^{E^{in}(T)} \oplus \bigotimes_{v \in V(E)} T \mathcal{B}_{|v|} \rightarrow L^p(\Omega^{(0,1)} (u^*(T(T^*M)))). $$
The terms $T \mathcal{B}_{|v|}$ correspond to the modular parameter of the vertex regions and are mapped to zero. The terms $\R^{E^{in}(T)}$ correspond to the modular parameter of the internal strips. We fix a generator of these moduli parameters assigning to any internal edge $e$ a compact support function $\chi_e$ on the edge. The infinitesimal deformation associated on the strip is the rescaling of the metric on the $s$ direction by $\chi_e$ (where we consider $\chi_e$ as a function independent on $t$).  


On each vertex region the operator $D$ is simply the standard $\bar \partial$ operator with values in $\C^n$.
On a strip $e$ it is the linearization of (\ref{strip}):
\begin{equation} \label{striplinear}
D^e(\xi_e,\lambda_e)=\nabla_s \xi_e + J(u) \nabla_t \xi_e - \varepsilon \lambda \tilde \chi_e \nabla f_e - \varepsilon \rho \nabla(\nabla f_e) \xi_e.
\end{equation}

The function $\tilde \chi_e$ comes from two contributions.  
The first is the derivative of $\partial_s u$ with respect to $\lambda_e$, that is $ \chi_e \partial_s u$. This becomes $\varepsilon \chi_0 \rho \nabla f$ after substituting the explicit expression of $u$ given by (\ref{stripsolution}).  
The second is the derivative of $\rho_l$ with respect to $l$. This is a nonnegative function because of our choice of the cutoff $\rho_l$. 
Therefore $\tilde \chi_e$ is non negative. In the following we only need that the integral of $\tilde \chi_e$ is not zero. 

\begin{lemma}
For $\varepsilon$ small enough
$$ \ker D \cong \ker D_0 \otimes \bigotimes_{v \in V(E)} T \mathcal{B}_{|v|} .$$
\end{lemma}
\begin{proof}
Suppose that $(\xi,\lambda)$ is in the kernel of $D$. As before on any strip we can split the equation into horizontal and vertical parts and prove that for $\varepsilon$ small enough $\xi$ is constant on all the vertex regions, depending only on $s$ along the strips and without vertical component. On any strip $e$ equation (\ref{striplinear}) becomes 
\begin{equation} \label{reduction}
\nabla_s \xi_e - \varepsilon \tilde \chi_e  \nabla f_e - \varepsilon \rho \nabla(\nabla f_e) \xi_e =0.
\end{equation}
where $\xi_e$  does not depend by $t$. As in $(\ref{tangent0})$
the space of solutions of (\ref{reduction}) on the edge $e$ is isomorphic to $TM_e$ . This can be seen as follows.
If $\xi_e^0$ is in the kernel of (\ref{morselinear}) with $\lambda=0$, the vector $\xi_e$ defined by $\xi_e(t,s)= \xi_e^0(l(s))$ solves (\ref{reduction}). If $\tilde \psi$ is a primitive of $\tilde \chi$,  $\tilde \psi (s)\dot \gamma_e( l(s))$ solves (\ref{reduction}) for $\lambda=1$. From this it is easy to construct the isomorphism. 

\end{proof}

By the last lemma, in order to conclude that $D$ is surjective it is enough to prove that the index of $D$ is the index of $D_0$ plus the sum of dimensions of $ \mathcal{B}_{|v|}$.
This follows since the index of $D$ is equal to the index of $D_0$ if one does not consider the modular parameters. 


\end{document}